\documentclass[reqno,11pt]{amsart}

\usepackage[T1,T2A]{fontenc}
\usepackage[utf8]{inputenc}

\usepackage{latexsym}
\usepackage{amssymb}
\usepackage{amsrefs}
\usepackage{bbm}
\usepackage[mathscr]{euscript}
\usepackage{nicefrac}

\usepackage{xcolor}
\usepackage{graphicx}
\usepackage{adjustbox}
\usepackage{caption}
\usepackage{tabularx}
\usepackage{tikz}
\usepackage{tikzscale}

\usetikzlibrary{
  calc,
  angles,
  quotes,
  positioning,
  automata,
  arrows,
  decorations.pathreplacing,
  fadings,
  3d
}

\usepackage[all]{xy}

\usepackage[normalem]{ulem}
\usepackage{blindtext}
\usepackage{comment}
\usepackage{stackengine}
\usepackage{enumitem}
\usepackage{subfiles}
\usepackage{xparse}

\stackMath

\usepackage{hyperref}

\NewDocumentCommand{\vdotss}{ O{0.35em} O{0.8pt} m }{%
  \begin{tikzpicture}[baseline={(0,-0.5ex)}]
    \pgfmathtruncatemacro{\lastdot}{#3-1}%
    \foreach \i in {0,...,\lastdot}{%
      \fill (0,{-\i*#1}) circle[radius=#2];
    }%
    \path (0,0.3ex) -- (0,{-\lastdot*#1-0.3ex});
  \end{tikzpicture}%
}

\newcommand{\eq}[2]{%
  \begin{equation}
    \label{eq:#1}
    #2
  \end{equation}%
}

\newcommand{\equ}[1]{\eqref{eq:#1}}

\newcommand{\hs}{homogeneous space}
\newcommand{\ggm}{G/\Gamma}
\newcommand{\hd}{Hausdorff dimension}
\newcommand{\ls}{Lagrange spectrum}

\newcommand{\da}{Diophantine approximation}
\newcommand{\di}{Diophantine}

\newcommand{\ehs}{expanding horospherical subgroup}

\newcommand{\ignore}[1]{}

\DeclareMathOperator{\Lie}{Lie}
\DeclareMathOperator{\dist}{dist}

\DeclareMathOperator{\GL}{GL}
\DeclareMathOperator{\SL}{SL}

\DeclareMathOperator{\Span}{Span}
\DeclareMathOperator{\Ad}{Ad}

\DeclareMathOperator{\diag}{diag}
\DeclareMathOperator{\vol}{vol}

\DeclareMathOperator{\lgr}{lgr}

\newcommand{\vp}{\mathbf{p}}
\newcommand{\vq}{\mathbf{q}}
\newcommand{\va}{\mathbf{a}}
\newcommand{\vb}{\mathbf{b}}

\newcommand{\p}{\mathbf{p}}

\newcommand{\vx}{\mathbf{x}}
\newcommand{\vy}{\mathbf{y}}

\newcommand{\R}{\mathbb{R}}
\newcommand{\Z}{\mathbb{Z}}
\newcommand{\N}{\mathbb{N}}

\newcommand{\nz}{\smallsetminus\{0\}}
\newcommand{\mr}{M_{m,n}(\R)}
\newcommand{\amr}{\ensuremath{{A}\in \mr}}

\newcommand{\vre}{\varepsilon}

\theoremstyle{plain}
\newtheorem{theorem}{Theorem}[section]
\newtheorem{lemma}[theorem]{Lemma}
\newtheorem{corollary}[theorem]{Corollary}

\newtheorem{proposition}[theorem]{Proposition}

\theoremstyle{definition}

\theoremstyle{remark}

\numberwithin{equation}{section}

\renewcommand{\le}{\leqslant}
\renewcommand{\ge}{\geqslant}

\DeclareFixedFont{\got}{U}{euf}{m}{n}{14.4pt}

\makeatletter
\newcommand{\namedlabel}[2]{%
  \begingroup
    #2%
    \def\@currentlabel{#2}%
    \phantomsection
    \label{#1}%
  \endgroup
}
\makeatother

\title[Density of the multidimensional Lagrange spectrum]{%
  Density of the multidimensional Lagrange spectrum%
}

\author[D. Kleinbock]{Dmitry Kleinbock}
\address{%
  Department of Mathematics,
  Brandeis University,
  Waltham, MA 02453, USA%
}
\email{kleinboc@brandeis.edu}

\date{}
\subjclass{11J06; 11J13, 37A17, 37A25}
\keywords{Diophantine approximation, Lagrange spectrum,  homogeneous dynamics, mixing, equidistribution of translates of horospheres}

\begin{document}

\begin{abstract} 
The Lagrange spectrum is a classical object in number theory, defined as the set of values of $ \liminf_{q\to\infty} q \dist(q\alpha,\Z)$ where $\alpha$ runs through irrational numbers. It has a  complicated structure, with the discrete part, Hall's ray, and a transitional part in between. One can similarly define the \ls\ in the multidimensional set-up, and until now   not much was understood about it. In this paper we prove that, unlike in the one-dimensional case, the closure of the multidimensional Lagrange spectrum is equal to the interval between $0$ and its supremum. The proof relies on a correspondence between \da\ and dynamics on the space of unimodular lattices  and proceeds by studying a dynamical counterpart of the \ls\ that we call dynamical \ls. The latter, under certain higher rank assumptions,  is shown to be equal to the interval between $0$ and its maximum.
% by means of an argument utilizing the higher rank nature of the set-up. 
A passage from the full dynamical spectrum to the density of the \di\ spectrum is achieved by applying equidistribution of expanding translates of horospheres in the space of lattices.
\end{abstract}

\maketitle

%\begin{sloppypar}

\section{Introduction}

\subsection{Basics of one-dimensional \da}
Let $\alpha$ be a real number. Dirichlet's theorem states that there are infinitely many positive integers $q$ such that
\[
|q\alpha-p| < \frac{1}{q}\ \text{ for some }p\in\Z.
\]
%holds, where $\|\cdot\|$ denotes the distance to the nearest integer. 
Hurwitz obtained a more precise result: for any irrational number $\alpha$, the inequality \linebreak
\(
|q\alpha-p| < \frac{1}{\sqrt{5}q}
\)
has infinitely many solutions $(p,q)\in\Z\times \N$. Moreover, there exists a countable set of numbers $\alpha$ for which this inequality is sharp, that is, for any positive $\varepsilon$ there are only finitely many $(p,q)\in\Z\times \N$ such that the inequality
\(
|q\alpha-p| < \left(\frac{1}{\sqrt{5}}-\varepsilon\right)/{q}
\)
holds.

Let us define the \textsl{Lagrange constant} $\lgr(\alpha)$ of a real number $\alpha$ as 
$$\lgr(\alpha) := \liminf_{q\to\infty} q \min_{p\in\Z}|q\alpha-p|  = \liminf_{q\to\infty} q \dist(q\alpha,\Z), $$
and call the image of the function $\lgr(\cdot)$ the \textsl{Lagrange spectrum} \eq{deflagrange}{\mathbb{L} := \{\lgr(\alpha) : \alpha\in\R\}.}   
Then one has $\mathbb{L} \subset\left[0, 1/{\sqrt{5}}\right]$.  It follows from Khintchine's theorem that $\lgr(\alpha) = 0$ for Lebesgue-almost every $\alpha\in\R$. A real number  whose Lagrange constant is positive is called \textsl{badly approximable}; those are known to form a set of \hd\ 1.  It is well-known that the set $\mathbb{L} $ has a quite complicated structure; in particular,  it contains a discrete part
\(
\frac{1}{\sqrt{5}}, \frac{1}{\sqrt{8}}, \dots
\)
accumulating at $\frac13$, and  also  an interval $[0, \lambda^*]$ known as Hall's ray. See \cite{Hall}, where it was shown that 
 $\lambda^* >0$, and Freiman's work  \cite{Freymanbook}, where   $$\lambda^*   = \frac{491993569}{2221564096 + 283748 \sqrt{462}} 
 \approx 0.22085637$$ was explicitly computed.  Between the discrete top part and Hall's ray $\mathbb{L}$ possesses a rich fractal structure, with the Hausdorff dimension transitioning from $0$ near $1/3$ to $1$ near $\lambda^*$, see \cite{Mor}. The book \cite{CF} is an excellent introduction to the topic; note that all of the aforementioned results are obtained using continued fractions. We remark that in the literature, such as in \cite{CF} or   \cite{Mor}, it has been   customary to refer to the set $$\{1/\lgr(\alpha) : \alpha\in\R\} \subset (0,\infty]$$ as the Lagrange spectrum; however it is  more convenient for us to use definition \equ{deflagrange}, following \cite{A, Mos, PP, AM} and some other recent papers.

\subsection{The multidimensional \ls}  \label{multi} It is natural to extend the notion of Lagrange constants and  the Lagrange spectrum to simultaneous \da. Fix $m, n \in \N$, and choose norms $\rho$ on $\R^m$ and  $\sigma$ on $\R^n$. Denote by $M_{m,n}(\R)$ the space of $m\times n$ matrices with real entries. For ${A} \in M_{m,n}(\R)$ (interpreted as a system of $m$ linear forms in $n$ variables) we define its \textsl{Lagrange constant} $\lgr_{\rho,\sigma}({A})$ by $$\lgr_{\rho,\sigma}({A}) := \liminf_{ \mathbf{q}\to \infty} \sigma(\mathbf{q})^n \min_{\vp\in\Z^m} \rho({A}\mathbf{q}-\vp)^m = \liminf_{ \mathbf{q}\to \infty} \sigma(\mathbf{q})^n \dist_\rho({A}\mathbf{q}, \Z^m)^m ,$$
where $\dist_\rho$ is the distance on $\R^m$ induced by the norm $\rho$. The \textsl{multidimensional Lagrange spectrum} is the range of this function:$$\mathbb{L}_{\rho,\sigma} := \{ \lgr_{\rho,\sigma}({A}) : {A} \in M_{m,n}(\R) \}.$$ When both $\rho$ and $\sigma$ are supremum norms, denoted by $\|\cdot\|_\infty$, we will use the notation $$\lgr_{m,n}({A}) =\liminf_{ \mathbf{q}\to \infty} \|\mathbf{q}\|_\infty^n \min_{\vp\in\Z^m} \|{A}\mathbf{q}-\vp\|_\infty^m$$ and $\mathbb{L}_{m,n} = \{ \lgr_{m,n}({A}) : {A} \in M_{m,n}(\R) \}$. Another important special case is given by $n=1$ and $\rho$ an arbitrary norm on $\R^m$; then for a vector $\va\in\R^m$  we will denote $$\lgr_{\rho,1}(\va) =\liminf_{ {q}\to \infty} |{q}|  \min_{\vp\in\Z^m} \rho(q\va-\vp)^m$$ and $\mathbb{L}_{\rho,1} = \{ \lgr_{\rho,1}(\va) : \va \in \R^m \}$.

When $m=n=1$,
% and the norms $\rho,\sigma$ on $\R$ are given by the absolute value,
one recovers the classical Lagrange spectrum $\mathbb{L} = \mathbb{L}_{1,1}$ defined in \equ{deflagrange}. Matrices ${A}$ with $\lgr_{m,n}({A}) > 0$ (equivalently, with $\lgr_{\rho,\sigma}({A}) > 0$  for any norms $\rho,\sigma$ on $\R^m$ and $\R^n$) are called {badly approximable}. The set of those matrices is denoted by $\mathbf{BA}_{m,n}$; it is known to be  of Lebesgue measure zero and Hausdorff dimension $mn$.

Let us denote  $L_{\rho,\sigma} = \sup \mathbb{L}_{\rho,\sigma} $,  $L_{\rho,1} = \sup \mathbb{L}_{\rho,1} $ and $L_{m,n} = \sup \mathbb{L}_{m,n} $. It follows from Minkowski's Convex Body Theorem that $L_{\rho,\sigma}$ is always finite; more precisely, $$L_{\rho,\sigma}\le \frac{2^d }{\vol(B_{\rho,\sigma})},$$ where $B_{\rho,\sigma}$ is the product of the $\rho$-unit ball in $\R^m$ and the $\sigma$-unit ball in $\R^n$. Specializing to    $\rho$ and $\sigma$ being supremum norms, one gets $L_{m,n}\le 1$, which is a consequence of Dirichlet's theorem in simultaneous \da. As is the case for $m=n=1$,  the upper bound obtained this way is in general not sharp. For example, it follows from 
\cite[Chapter 2, Theorem 3A]{S} that  $$L_{m,n}\le  \frac{m^m n^n}{(m+n)^{m+n}} \cdot \frac{(m+n)!}{m! n!}.$$
However, unless $ m = n = 1$, the exact value of $L_{m,n}$ is not known, and it is not clear whether or not $L_{m,n}$ belongs to  $\mathbb{L}_{m,n}$. 
For norms other than supremum norms the 
%above questions are also open, except 
only result known to the author about the top of the spectrum is for the case $\{m,n\} = \{1,2\}$ and $\rho$ the Euclidean norm on $\R^2$. 
In this case both ${L}_{\rho,1}$ and ${L}_{1,\rho}$ were proved by Davenport and Mahler \cite{DM}  to be equal to $2/\sqrt{23}$, and it 
is mentioned in \cite[footnote on p.\ 173]{Cas} that using "isolation" techniques one can   show  that  ${L}_{\rho,1}\notin\mathbb{L}_{\rho,1}$ and ${L}_{1,\rho}\notin\mathbb{L}_{1,\rho}$.
%can be shown using "isolation" techniques based on \cite{DR} that ${L}_{\rho,1}$ does not belong to $\mathbb{L}_{\rho,1}$, see \cite[footnote on p.\ 173]{Cas}. 
%Also to the best of the author's knowledge it is not  known whether or not  the multidimensional Lagrange spectrum is uncountable. 

More to the topic of this work, until the present paper it was not known if in higher dimensions the complement 
%if  the closure of 
$[0,L_{\rho,\sigma}]\smallsetminus\mathbb{L}_{\rho,\sigma} $ can contain an interval. 
There have been only some partial results on relative density of $\mathbb{L}_{\rho,1}$ for $m > 1$. Namely, in \cite[Corollary 3]{A}, improving on an earlier work \cite{AM},  it was shown that 
\eq{akhmos}{
\begin{aligned}%\text{for all } 
\forall\,m > 1   %text{ there exist }
 \text{ and an arbitrary norm $\rho$ on }\R^m\ 
\exists\, r_0,\kappa > 0\text{  such that for }0 < r \le r_0,\\ \text{the closed interval }\left[r,r\left(1 + \kappa r^{\frac{m+1}{m^2}}\right)\right]\text{  contains an element of }\mathbb{L}_{\rho,1}.\qquad\end{aligned}} %In a sequel \cite{A} to \cite{AM} the aforementioned result was extended to an arbitrary  norm $\rho$ on $\R^m$, and the dependence of $r$ and $\kappa$ on $m$ and  the norm was made explicit. 
%Finally, v
Very recently Ward \cite[Theorem  1.1]{W}  %announced a proof of 
took $\rho$ to be   the Euclidean norm on $\R^m$ and strengthened \equ{akhmos}  by replacing the interval $\left[r,r\left(1 + \kappa r^{\frac{m+1}{m^2}}\right)\right]$ with $\left[r,r(1 + \kappa r)\right]$, getting the same conclusion. 
%From there it easily follows that the lower box dimension of $\mathbb{L}_{m,1}$ is at least $1 - \frac1{m+1}$. 

\smallskip
The first main result of this paper is a significant improvement and generalization of the aforementioned developments. Namely we  prove

\begin{theorem}\label{thm1} Let $m, n \in \N$ with $\max(m,n) > 1$ be given, and let  $\rho$  and  $\sigma$ be   norms on $\R^m$ and $\R^n$ respectively. Then  the closure of  ${\mathbb{L}_{\rho,\sigma}} $ is equal to  $[0, L_{\rho,\sigma}]$. Moreover,  for any non-empty open interval $I \subset \left[0, L_{\rho,\sigma}\right]$, the set 
$$  \{A\in\mr : \lgr_{\rho,\sigma}(A) \in I\}$$
is dense in $\mr$.

\end{theorem}

This shows  that the multidimensional \ls\ behaves very differently from its one-dimensional counterpart. Note that our proof sheds no new light on whether or not $L_{\rho,\sigma}$ itself belongs to the spectrum, and also does not imply that ${\mathbb{L}_{\rho,\sigma}} $ is uncountable. We remark, however, that Ward made a rather elementary observation at the end of his paper \cite[Theorem 7.1]{W} establishing that ${\mathbb{L}_{n,n}} \supset \{c^n: c\in \mathbb{L}\}$ for any $n\in\N$; in particular ${\mathbb{L}_{n,n}}$ contains  a Hall’s ray. Most recently, in a preprint by Cheung, Wang and Zhang \cite{CWZ} it was shown that $\mathbb{L}_{\rho,1}$ contains  a Hall’s ray for any norm $\rho$ on $\R^2$.

%where $\Vert{}\cdot\Vert{}$ denotes the supremum norm on the respective Euclidean spaces (or any fixed norm, up to a constant shift). 

\subsection{A scheme of proof via a dynamical \ls}
Following the seminal framework of Dani \cite{dani}, we translate the computation of multidimensional Lagrange constants into the language of trajectories in the space of unimodular lattices in $\R^{m+n}$. More abstractly, we prove a dynamical statement that holds in a wider generality of flows on noncompact  \hs s. 

Let $G$ be a connected Lie group, $\Gamma \subset G$ a non-uniform lattice, and $X = G/\Gamma$ the corresponding homogeneous space equipped with its unique $G$-invariant probability measure $\mu$.  Let $\phi : X \to \R_{> 0}$ be a continuous, proper, bounded function, let  $\{g_t: t \in \R\} \subset G$ be a one-parameter subgroup, and let $F = \{g_t: t \ge 0\}$. For any $x\in X$ define the \textsl{dynamical Lagrange constant} of $x$ with respect to $F$ and $\phi$ by 
\eq{lambdadyn}{
\lgr^F_\phi(x):= \liminf_{t\to\infty}\phi(g_tx).}
One can say that $\lgr^F_\phi(x)$  measures the "essential size" of the trajectory $Fx$: if  $\lgr^F_\phi(x) = \varepsilon$, then the compact set $\{\phi \ge \varepsilon\}\subset X$ is the smallest super-level set of $\phi$  containing all the forward limit points of $Fx$. In particular, $\lgr^F_\phi(x) = 0$  if and only if $Fx$ is unbounded. Hence, if in addition we know that the $F$-action on $(X,\mu)$ is ergodic, then $\lgr^F_\phi(x) = 0$ for $\mu$-a.e.\ $x\in X$. 

We refer to the set of values of $\lgr^F_\phi$ as the \textsl{dynamical Lagrange spectrum} with respect to $F$ and $\phi$:
\eq{Ldyn}{
\mathbb{L}^F_\phi := \left\{\lgr^F_\phi(x): x\in X\right\}.
}
More generally, for any $Y\subset X$ one can define the  {spectrum} restricted to $Y$: $$
\mathbb{L}^F_\phi(Y) := \left\{\lgr^F_\phi(x): x\in Y\right\}.
$$ 
Let $L^F_\phi:= \sup \mathbb{L}^F_\phi$; the boundedness of $\phi$ implies that $L^F_\phi < \infty$. Also it is known that for any $G/\Gamma$ and $F$ there exist $x\in X$ with bounded $F$-trajectories, see \cite[Corollary 4.1.7]{KSS}; hence one always has $L^F_\phi > 0$.

% is finite, and, in view of the positivity and properness of $\phi$, one has  $\lgr^F_\phi(x) = 0$  if and only if $Fx$ is unbounded. 

As an example, consider the dynamical system that encodes the \da\ properties of matrices \amr. Take $G = \SL_d(\R)$ and $\Gamma = \SL_d(\Z)$, where $d = m + n$.  The homogeneous space $X_d := G/\Gamma$ naturally parametrizes the space of unimodular lattices in $\R^d$ via the map $g\Gamma \mapsto g\Z^d$. 
For a pair of norms $\rho$ on $\R^m$ and $\sigma$ on $\R^n$,   define a norm $\|\cdot\|$ on $\R^d$ by 
\eq{defeta}{
\left\|\begin{pmatrix}\mathbf{y}\\ \mathbf{z}\end{pmatrix}\right\| := \max\big( \rho(\mathbf{y}), \, \sigma(\mathbf{z}) \big)}
where $\begin{pmatrix}\mathbf{y}\\ \mathbf{z}\end{pmatrix}\in \R^m \times \R^n = \R^d$. 
Then define   $\phi : X_d \to \R_{>0}$ as the $d$-th power of the  norm of the shortest non-zero vector in a lattice ${x} \in X_d$:
%. Let $\|\cdot\|_{\rho,\sigma} $ be a norm on $\R^d$ given by 
%Then, for a lattice ${x} \in X_d$, we define the minimum vector length function $\delta_{\eta}$ by
\eq{defdelta}{
\phi({x}) := \inf_{\mathbf{v} \in {x} \nz} \|\mathbf{v}\|^d.
}
%. Let $\|\cdot\|_{\rho,\sigma} $ be a norm on $\R^d$ given by 
%Then, for a lattice ${x} \in X_d$, we define the minimum vector length function $\phi$ by
%\eq{defdelta}{
%\phi({x}) := \inf_{\mathbf{v} \in {x} \nz} \|\mathbf{v}\|^d.
%}
It is clearly continuous in view of the continuity of the $G$-action on $\R^d$, bounded (Minkowski's Convex Body Theorem) 
%, see e.g. \cite[Ch. III,   Theorem II]{Cas})
and proper  (Mahler's Compactness Criterion, see e.g.\ \cite{Cas}).
% \cite[Theorem IV, \S V.4.2]{Cas}).
%that the function $\phi$ is bounded and tends to $0$ at infinity in $X_d$. 
Also because every  $x \in X_d$ is a discrete subset of  $\R^d$, the infimum in \equ{defdelta} is always attained at some non-zero vector of $x$.

Let $I_k$ stand for the $k\times k$ identity matrix. Define the one-parameter diagonal subsemigroup 
\eq{defF}
{F = \{g_t: t \ge 0\} \subset G, \quad \text{where }g_t := \begin{pmatrix} e^{t/m} I_m & 0 \\ 0 & e^{-t/n} I_n \end{pmatrix},
}
and, given an $m \times n$ matrix $A \in M_{m,n}(\R)$,  associate to it the unimodular matrix
\[
u_A := \begin{pmatrix} I_m & A \\ 0 & I_n \end{pmatrix} \in G
\]
and the corresponding unimodular lattice $x_A := u_A \Z^d \in X_d$. Note that the group \eq{defU}{U:= \{u_A: A\in\mr\}} is %precisely 
the expanding horospherical subgroup of $G$ relative to $F$ (see \equ{defUgeneral} for a definition); this fact will be crucial for our approach.

%The non-zero vectors of $x_A$ are precisely of the form
%\[
%\mathbf{v}_{\mathbf{p}, \mathbf{q}} = \begin{pmatrix} A\mathbf{q} - \mathbf{p} \\ \mathbf{q} \end{pmatrix} \quad \text{for } (\mathbf{p}, \mathbf{q}) \in (\Z^m \times \Z^n )\nz .
%\]
In  \S\ref{dioph} we will prove the following identity %connecting $\lgr_{\rho,\sigma}(A)$ to the asymptotic behavior of the trajectory $\{g_t x_A\}_{t \ge 0}$ in $X_d$ 
that 
can be thought of as a quantitative strengthening of Dani's Correspondence: 

\begin{proposition}\label{thm:dynamical_formula}
For any   $A \in M_{m,n}(\R)$ and any pair of norms $\rho, \sigma$ on $\R^m$ and $\R^n$ respectively, we have
\eq{dynformula}{
\lgr_{\rho,\sigma}(A) = %\liminf_{t \to \infty}   \phi(g_t x_A)
\lgr^F_\phi(x_A),
}
where $F$ is as in \equ{defF}, $\phi$ is as in \equ{defdelta} and the norm $\|\cdot\|$ is as in 
\equ{defeta}. Consequently, $ \mathbb{L}_{\rho,\sigma} =   \mathbb{L}^F_\phi(U\Z^d)$, where $U$ is as in \equ{defU}.
\end{proposition}

\subsection{The main dynamical results} Our next theorem ensures that, under certain "higher rank" assumptions on the system, the full dynamical \ls\ $\mathbb{L}^F_\phi$ on $X = \ggm$ coincides with $[0,L^F_\phi]$. An important role in the statement will be played by the Lie subalgebra  of $ \mathfrak{g} =  \Lie(G)$ 
%of directions along which the adjoint action of $g_t$ remains bounded, 
and the corresponding Lie subgroup  of $G$ that encompass the neutral directions with respect to the  action of $\{g_t: t\in \R\}$:
%\begin{itemize}
%\item the stable directions with respect to the $F$-action:
\eq{defZ}{
\begin{aligned}
\mathfrak{z} &
:= \big\{ Y \in \mathfrak{g} : %\sup_{t \ge 0} \|
\{\operatorname{Ad}(g_t) Y: %\| < \infty
t \in\R\}\text{ is bounded}  \big\}, \text{ and} \\ Z &
:= \text{ the connected Lie subgroup with } \Lie(Z) = \mathfrak{z}.
\end{aligned}
}
%\item the unstable directions with respect to the $F$-action:
%\eq{defUgeneral}{
%\begin{aligned}
%\mathfrak{u} &:= \left\{ Y \in \mathfrak{g} :  \operatorname{Ad}(g_{-t})Y\to 0 \text{ as }t\to\infty \right\}\text{ and} \\ U &\ = \text{ the expanding horospherical subgroup of $G$ relative to }F\\ &:= \text{ the connected Lie subgroup with } \Lie(U) = \mathfrak{u} \\ &\ = \ \{u \in G : g_{-t} u g_t \to e \text{ as } t \to \infty\}.
%\end{aligned}
%}
%\end{itemize}
Then $Z$ is the connected subgroup of $G$ maximal with respect to the property that \linebreak $\{g_t {h} g_{-t}: t \in\R\}$ is bounded for any ${h}\in Z$.
%, and $U$ is the expanding horospherical subgroup of $G$ relative to $F$.

We remark that, assuming  $g_t = \exp(t Y_0)$
% the  subalgebra $\mathfrak{u}$  coincides with the sum of all generalized eigenspaces of $\operatorname{ad}_{Y_0}$ corresponding to eigenvalues $\lambda \in \mathbb{C}$ with $\operatorname{Re}(\lambda) > 0$. Additionally, 
%when $g_t $ 
is $\operatorname{Ad}$-diagonalizable over $\mathbb{C}$, the subalgebra $\mathfrak{z}$  coincides with the sum of all generalized eigenspaces of $\operatorname{ad}_{Y_0}$ corresponding to  $\operatorname{Re}(\lambda) = 0$; hence  in this case %$\mathfrak{g} = \mathfrak{s}\oplus \mathfrak{u}$, and 
$G$ is locally a direct product of $Z$ with the expanding and contracting horospherical subgroups of $G$ relative to $F$, namely
\eq{defUgeneral}{U   := \{u \in G : g_{-t} u g_t \to e \text{ as } t \to \infty\}\text{ and }U^- := \{u \in G : g_t u g_{-t} \to e \text{ as } t \to \infty\}.} 
See \cite[Definition 2.1]{EKR} for more details.
%When $F$ is $\operatorname{Ad}$-diagonalizable over $\mathbb{R}$, $Z$ coincides with the identity component of the centralizer of $F$ in $G$. In general, $\mathfrak{z}_F$ consists of the semi-simple component of the neutral Lie algebra.

%\eq{defZ}{
%Z:= \text{ the identity component of the centralizer of $F$ in  }G.
%}
%$Z$ the connected component of the centralizer of $F$ in $G$.

\begin{theorem}\label{thm:full_spectrum_general} 
Let $G$ be a  connected Lie group, $\Gamma \subset G$ a lattice, and $X = G/\Gamma$. Let $\phi: X \to \R_{> 0}$ be a  bounded, proper, continuous  function,  let $F = \{g_t: t \ge 0\} \subset G$ be a %non-trivial 
one-parameter subsemigroup, and let $Z$ be as in \equ{defZ}.  Assume that %the following two conditions hold:
%\begin{itemize}
%\item[\rm (i)] $\log \phi$ is uniformly continuous along $Z$, i.e.\ for every $\varepsilon > 0$ there exists an open neighborhood $B$ of the identity  in $Z$ such that
%\eq{subhom}{
%\phi(z x) \le (1 + \varepsilon) \phi(x) \quad \text{for all } x \in X \text{ and } z \in B;
%}
%(equivalently, )
%\item[\rm (ii)] 
\eq{unbdd}{\text{for any $x\in X$ the orbit $Z x$ is unbounded.}}
%\end{itemize}
Then 
\(
\mathbb{L}^F_\phi \supset \left(0,L^F_\phi\right]\). Consequently, if in addition one knows that at least one unbounded $F$-trajectory exists (for example, if the $F$-action  on $(X,\mu)$ is ergodic), then \(
\mathbb{L}^F_\phi = \left[0,L^F_\phi\right]\).

\end{theorem}

Specializing to $X = X_d$, %it is easy to see that $\phi$ as in \equ{defdelta} satisfies (i), in fact $\log \phi$ is uniformly continuous in all directions (in the terminology of \cite{KM99}, $\log(1/\phi)$ is a distance-like function on $X_d$), and (ii) 
we will show in \S\ref{sld} that \equ{unbdd} is satisfied for $F$ as in \equ{defF} as long as $d\ge 3$; since the action of any unbounded semigroup of $\SL_d(\R)$ on $X_d$ is ergodic, one can conclude that in this case the dynamical \ls\ $\mathbb{L}^F_\phi$ coincides with the interval $ \left[0,L^F_\phi\right]$. 

However recall that for \di\ applications we are interested in the restriction of $\mathbb{L}^F_\phi$ to $U\Z^d$, where $U$ is as in \equ{defU}. This is achieved by means of the next theorem:

\begin{theorem}\label{thm:dynamical_density_general}
Let $X = G/\Gamma$, $F = \{g_t: t \ge 0\}$ and $\phi$ be as in Theorem \ref{thm:full_spectrum_general} (in particular, \equ{unbdd} is satisfied), and assume in addition that $F$ is $\Ad$-diagonalizable over $\mathbb{C}$ and the $F$-action on $(X,\mu)$ is mixing.  Let $U$ be %the expanding horospherical subgroup relative to $F$.
as in \equ{defUgeneral}. Then for any $x \in X$ and any non-empty open $W\subset U$, the restricted spectrum $\mathbb{L}^F_\phi(Wx)$ is dense in $[0, L^F_\phi]$.
% containing $\lgr^F_\phi(x_0)$, the set
%\[
%\left\{ u \in U : \lgr^F_\phi(ux) \in I \right\}
%\]
%is dense in $U$.
\end{theorem}

Since the action of $F$ as in \equ{defF} on $X_d$ is mixing \cite{Moo} and in view of Proposition \ref{thm:dynamical_formula},
%in view of the Howe--Moore theorem [reference]
 Theorem \ref{thm1}  follows. Moreover, with $\phi$ defined by $\rho,\sigma$ as in  \equ{defeta}--\equ{defdelta}, one concludes that $L_{\rho, \sigma} = L^F_\eta $, and that the closure of $
\mathbb{L}_{\rho, \sigma}$ is equal to $
\mathbb{L}^F_\phi $. 

 \medskip 
The paper is organized as follows: we prove Theorems \ref{thm:full_spectrum_general}  and \ref{thm:dynamical_density_general} in \S\ref{dynlagr}, and then, in \S\ref{dioph}, provide a proof of  Proposition \ref{thm:dynamical_formula} and finish the proof of 
Theorem~\ref{thm1}.  
The last section of the paper lists several open questions.

 \medskip 
\noindent{\bf Acknowledgements.} This work was conceived when the author was visiting the Banach Center in Warsaw, Poland during a special semester on Continued Fractions, Fractal Geometry, Ergodic Theory and Dynamics; the hospitality of the center is gratefully acknowledged. Thanks are also due to  %Yitwah Cheung, 
Konstantin Andritsch, Nikolay Moshchevitin, Vasiliy Neckrasov and Benjamin Ward for helpful discussions.
 
 \medskip 
\noindent{\bf Use of AI.} Google Gemini and ChatGPT Pro were used as auxiliary tools
in checking selected calculations and arguments   and  refining the exposition.
The author independently verified every mathematical argument, determined
the final content and wording, and takes full responsibility for the manuscript.

%\section{Quantitative Dani correspondence}

\section{Dynamical Lagrange spectrum}\label{dynlagr}

In this section our goal is to prove Theorems \ref{thm:full_spectrum_general}  and \ref{thm:dynamical_density_general}. 
Let $X = G/\Gamma$ be a homogeneous space of finite volume,  let $\phi : X \to \R_{> 0}$ be a bounded, continuous, proper function, and let $F = \{g_t: t \ge 0\} \subset G$ be a one-parameter subsemigroup.

\subsection{Full dynamical \ls}\label{1.3}
Our first observation, not requiring additional assumptions on $\phi$ and $F$,  is that the supremum of the dynamical \ls\ is always attained. 

\begin{proposition}\label{topspectrum} $L^F_\phi \in \mathbb{L}^F_\phi$.
\end{proposition}
 
\begin{proof}
Let $L := L^F_\phi = \sup_{x \in X} \lgr^F_\phi(x)$. %If $L = 0$ the conclusion is immediate. Otherwise,   
For each $k \in \mathbb{N}$ choose $x_k \in X$ such that $$\lgr^F_\phi(x_k) > L - \frac{1}{2k},$$ and then $t_k > 0$ such that 
\begin{equation} \label{eq:uniform_bound}
\phi(g_t x_k) \ge L - \frac{1}{k} \quad \text{for all } t \ge t_k.
\end{equation}
Define $x'_k := g_{t_k} x_k \in X$. Since $L - 1/k$ is positive for large enough $k$ and  $\phi$ is proper, the condition $\phi(x'_k) \ge L - 1/k$ implies that %for large enough $k$ 
the points $x'_k$ lie in a compact subset of $X$. Passing to a subsequence, we can assume that $x'_k \to x_0 \in X$. Since both $\phi$ and  the $F$-action on $X$ are continuous,    for any fixed $t \ge 0$ one has 
$$\phi(g_t x_0) = \lim_{k \to \infty} \phi(g_t x'_k) = \lim_{k \to \infty} \phi(g_{t+t_k} x_k) \underset {\eqref{eq:uniform_bound}}\ge L.$$
Taking $\liminf_{t \to \infty}$ yields $\lgr^F_\phi(x_0) \ge L$, therefore $\lgr^F_\phi(x_0) = L$.
\end{proof}

Next, we let $Z$ be %the centralizer of $F$ 
as in \equ{defZ}
and show that %under the assumption of uniform continuity of $\log \phi$ along $Z$-orbits  the function 
$\lgr^F_\phi$ is continuous along $Z$-orbits.

\begin{lemma}\label{ZF} %Suppose that condition {\rm (i)}  of Theorem \ref{thm:full_spectrum_general} holds. Then t
The function ${h}\mapsto \lgr^F_\phi({h}x)$ is continuous on $Z$ for every $x\in X$.
\end{lemma} 

\begin{proof} Clearly, by replacing $x\in X$ with ${h}x$, it suffices to prove the continuity of the function ${h} \mapsto \lgr^F_\phi({h}x)$ at $e \in Z$. Fix $x \in X$ and $\varepsilon > 0$. Let $K = \{x\in X: \phi(x)  \ge \varepsilon\}$, which is a compact subset of $X$. Since $G$ acts continuously on $X$ and $K$ is compact, $\phi$ is uniformly continuous on $K$ with respect to the $G$-action: specifically, there exists a symmetric neighborhood $W \subset G$ of the identity $e$ such that 
\begin{equation}\label{eq:unif_cont}
|\phi(g y) - \phi(y)| < \varepsilon \quad \text{whenever } y \in K \text{ and } g \in W.
\end{equation}

Now, %recall that $S = \exp(\mathfrak{s})$ where $\mathfrak{s} = \{ Y \in \mathfrak{g} : \sup_{t \ge 0} \|\operatorname{Ad}(g_t) Y\| < \infty \}$. 
using \equ{defZ} and the finite-dimensionalty of $\mathfrak{z}$, we conclude that  the operator norm of $\operatorname{Ad}(g_t)$ restricted to $\mathfrak{z}$ is uniformly bounded over all $t \in\R$ by some $M < \infty$.
%$$M := \sup_{t \ge 0} \left\Vert{} \operatorname{Ad}(g_t)\big\vert{}_{\mathfrak{z}} \right\Vert{}_{\mathrm{op}} < \infty$$
Thus we have the uniform estimate \eq{opnormest}{\Vert{}\operatorname{Ad}(g_t) Y\Vert{} \le M \Vert{}Y\Vert{}\quad\text{for all }Y\in \mathfrak{z} \text{ and }t \ge 0,}
where $\|\cdot\|$ is any norm on $\mathfrak{z}$. Since the exponential map $\exp: \mathfrak{g} \to G$ is a  diffeomorphism from a neighborhood of $0 \in \mathfrak{g}$ onto a neighborhood of $e \in G$, \equ{opnormest} implies that there exists an open  neighborhood of identity $B \subset Z$  such that 
$$ {h}_t := g_t {h} g_t^{-1} \in W \quad \text{for all } {h} \in B \text{ and all } t \ge 0. $$

For any $t \ge 0$ and ${h} \in B$,  write $g_t {h} x = (g_t {h} g_t^{-1}) g_t x = {h}_t (g_t x)$, and consider two cases:
%s for the point $g_t x \in X$:

\begin{enumerate}
    \item If $g_t x \in K$ or ${h}_t g_t x \in K$, then applying \eqref{eq:unif_cont} to $y = g_t x$ (or using ${h}_t^{-1} \in W$ when ${h}_t g_t x \in K$) yields:
    $$ |\phi(g_t {h} x) - \phi(g_t x)| = |\phi({h}_t g_t x) - \phi(g_t x)| < \varepsilon. $$
    
    \item If both $g_t x \notin K$ and ${h}_t g_t x \notin K$, then by definition of $K$ we have $$\max\big(\phi(g_t x), \phi(g_t {h} x)\big) < \varepsilon \quad\Longrightarrow \quad|\phi(g_t {h} x) - \phi(g_t x)| <   \varepsilon.$$ 
    %hence, by the triangle inequality,
 %   $$ |\phi(g_t h x) - \phi(g_t x)| \le \phi(g_t h x) + \phi(g_t x) < \frac{\varepsilon}{2} + \frac{\varepsilon}{2} = \varepsilon. $$
\end{enumerate}

Thus, $|\phi(g_t {h} x) - \phi(g_t x)| < \varepsilon$ holds for all $t \ge 0$ whenever ${h} \in B$. Taking $\liminf_{t \to \infty}$ on both sides gives
$$ \left| \lgr^F_\phi({h}x) - \lgr^F_\phi(x) \right| \le \varepsilon \quad \text{for all } {h}\in B. $$
This shows the continuity of ${h} \mapsto \lgr^F_\phi({h}x)$ at $e \in Z$ and finishes the proof of the lemma.
%. Since $h_0 \mapsto \lgr^F_\phi(h_0 h x)$ follows by translating $x$ to $h_0 x$, the function is continuous on all of $Z$. 
\end{proof}

Now we will  assume \equ{unbdd} and prove that all intermediate values of $\lgr^F_\phi$ are attainable along $Z$-orbits.

\begin{proposition} \label{movingdown}  Suppose that %conditions {\rm (i)}  and  {\rm (ii)}  of Theorem \ref{thm:full_spectrum_general} 
\equ{unbdd} holds.
%, and let $x\in X$ be such that $Z x$ is unbounded. 
Then for any $x\in X$  and any $c \in \left(0, \lgr^F_\phi(x)\right]$ there exists $h \in Z$ such that $\lgr^F_\phi(hx) = c$.
\end{proposition}

\begin{proof}
Take $x \in X$; the conclusion is immediate if $\lgr^F_\phi(x) = 0$, so assume that $$c_0 := \lgr^F_\phi(x) > 0.$$ Choose a sequence $t_k \to \infty$ such that $\phi(g_{t_k}x) \to c_0$ as $k \to \infty$. Since $\phi$ is proper, the sequence $g_{t_k}x$ has a limit point $x_\infty \in X$; by passing to a subsequence, we can assume that 
\eq{lim1}{x_\infty = \lim_{k\to\infty} g_{t_k}x.} 
Note that $\phi(x_\infty) = c_0$.

Now fix $c \in (0, c_0]$. By assumption \equ{unbdd}, the orbit $Z x_\infty$ is unbounded in $X$. Since $\phi$ is proper, the set $\{y \in X : \phi(y) \ge c\}$ is compact, thus we can choose an element $h' \in Z$ such that $\phi(h' x_\infty) < c$. Hence, by \equ{lim1} and the continuity of $\phi$ and the $G$-action on $X$, 
\eq{lim2}{\lim_{k\to\infty} \phi\left (h'g_{t_k}x\right) = \phi\left(\lim_{k\to\infty} h'g_{t_k}x\right) = \phi\left(h' \lim_{k\to\infty} g_{t_k}x\right) = \phi(h'x_\infty) < c.}

By definition of $\mathfrak{z}$, the family of operators $\big\{\operatorname{Ad}(g_t)|_{\mathfrak{z}} : t \in \mathbb{R}\big\}$ is bounded, so its closure in $\mathrm{Aut}(\mathfrak{z})$ is compact. Passing to a subsequence if necessary, we can conclude that the linear operators $\Phi _k := \operatorname{Ad}(g_{-t_k})|_{\mathfrak{z}}$ converge to an automorphism $\Phi \in \mathrm{Aut}(\mathfrak{z})$. Further, since operator inversion on finite-dimensional vector spaces is coontinuous, we obtain
\eq{lim3}{\Phi^{-1} _k \to \Phi^{-1}\text{ as }k\to\infty.}

Recall that $Z$ is connected; hence $h'$ can be written in the form 
$h' = \exp(Y_1)·\cdots\exp(Y_r)$ where  $Y_1,\dots,Y_r  \in \mathfrak{z}$.
Define $h'' := \exp(\Phi Y_1)\cdots  \exp (\Phi Y_r)\in Z$; in view of \equ{lim3},   the conjugation identity for exponentials gives
$$
g_{t_k} h'' g_{-t_k} = \exp (\Phi^{-1} _k \Phi Y_1)\cdots\exp (\Phi^{-1} _k \Phi Y_r)\to  \exp(Y_1)·\cdots\exp(Y_r) = h'.$$
Now, using \equ{lim1}, \equ{lim2}, and the continuity of $\phi$ and of the $G$-action on $X$, we can write:
$$ \phi(g_{t_k} h'' x) = \phi(g_{t_k} h'' g_{-t_k} g_{t_k} x) \xrightarrow{k\to\infty} \phi(h' x_\infty) < c. $$
Hence  $\lgr^F_\phi(h''x) \le \lim_{k\to\infty} \phi(g_{t_k} h'' x) = \phi(h' x_\infty) < c$.

Finally, since the connected Lie subgroup $Z$ is path-connected, choose a continuous path $(h_s)_{s \in [0, 1]} \subset Z$ with $h_0 = e$ and $h_1 = h''$. The function $f(s) := \lgr^F_\phi(h_s x)$ is continuous on $[0, 1]$ in view of Lemma \ref{ZF}, with $f(0) = \lgr^F_\phi(x) = c_0$ and $f(1) = \lgr^F_\phi(h'' x) < c$. By the Intermediate Value Theorem applied to $f$, there exists $s \in [0, 1]$ such that $$f(s) = \lgr^F_\phi(h_s x) = c,$$ completing the proof.
\end{proof}

We can now finish the 

\begin{proof}[Proof of Theorem \ref{thm:full_spectrum_general}]
%By Moore's Ergodicity Theorem (or ergodicity of $F$ on $X = G/\Gamma$), $0 \in \mathbb{L}^F_\phi$. 
By Proposition \ref{topspectrum}, there exists $x_0 \in X$ with $\lgr^F_\phi(x_0) = L^F_\phi$.  \linebreak
Applying Proposition \ref{movingdown} to $x_0$ shows that $\lgr^F_\phi(Z x_0) \supset  \left(0, L^F_\phi\right]$, as desired.
\end{proof}
 
\subsection{Density of the restricted spectrum}\label{1.4}
 In order to derive Theorem  \ref{thm:dynamical_density_general} from Theorem~\ref{thm:full_spectrum_general}, consider the  expanding and contracting horospherical subgroups $U$ and $U^-$ %relative to $F$ 
 defined in \equ{defUgeneral}.
 % $$U   := \{u \in G : g_{-t} u g_t \to e \text{ as } t \to \infty\}\text{ and }U^- := \{u \in G : g_t u g_{-t} \to e \text{ as } t \to \infty\}$$ be . 
 Then, under the assumption of  the $\Ad$-diagonalizability of $F$ over $\mathbb{C}$,   $G$ is locally the product of $U$, $U^-$ and $Z$.  In other words, one can choose neighborhoods of identity $B$, $B^-$ and $B^0$ in $U$, $U^-$  and $Z$ respectively such that the multiplication map \eq{diffeo}{B \times B^-\times B^0\to G} %(in any order)
  is a diffeomorphism onto a neighborhood of identity in $G$.
 \smallskip
 
 It is a trivial observation that $\lgr^F_\phi$ is $g_t$-invariant. The next lemma shows that the function $\lgr^F_\phi$  is  also $U^-$-invariant:

\begin{lemma}\label{uminus} For any $x\in X$ and any $u\in U^-$ one has $\lgr^F_\phi(ux) = \lgr^F_\phi(x)$. \end{lemma} 

\begin{proof} %Clearly $\lgr^F_\phi(x) = 0 \iff $, meaning that 
Since  for $u \in U^-$ one has $g_t u g_{-t} \to e$ as $t \to \infty$, it follows that $$\dist_X(g_t ux, g_t x) = \dist_X\big((g_t ug_{-t})g_tx, g_t x\big) \to 0\text{ as }t \to \infty.$$ 
(Here $\dist_X$ is the quotient metric induced by a right-invariant metric on
$G$.) Hence the unboundedness of  $Fx$ implies that of $Fux$, yielding $\lgr^F_\phi(ux)  = \lgr^F_\phi(x)  = 0$. And when  $Fx$ is bounded,  the uniform continuity of $\phi$ on compact sets gives  $$|\phi(g_t ux) - \phi(g_t x)| \to 0\text{  as }t \to \infty;$$ taking $\liminf_{t \to \infty}$ yields $\lgr^F_\phi(ux) = \lgr^F_\phi(x)$.
\end{proof}

We now focus our attention on the \ehs\ $U$. 
The key property that we are going to use is  the equidistribution of $g_t$-translates of $U$-orbits   in $X$ as $t\to\infty$. This is a well-known consequence of mixing of %the $F$-action on $(X,\mu)$
flows on \hs s, an idea dating back to the Ph.D.\ Thesis of Margulis \cite{M}, see also \cite{KM}. Specifically we shall use the following proposition

\begin{proposition} \label{equidistr}  Let $X = G/\Gamma$ and  $F = \{g_t: t \ge 0\}$ be as in Theorem \ref{thm:dynamical_density_general}, and let $U$ be  the expanding  horospherical subgroup relative to $F$. 
Then for any  $x \in X$ and any bounded non-empty open sets $B \subset U$ and $Q \subset X$  %of positive Haar measure and 
with $\mu(\partial Q) =0$    there exists $t_0 >0$ such that $g_tBx \cap Q\ne \varnothing$ for all $t \ge t_0$. 
\end{proposition}

\begin{proof} One can for example refer to \cite[Proposition 2.4]{Kl}:
%, under the assumption that the $g_t$-action is mixing  and 
for any $B,Q$ as above, any compact subset $L$ of $X$  and any $\varepsilon > 0$  it establishes the existence of $t_0 := t_0(B, Q, L, \varepsilon) >0$ such that for all $t > t_0$ and any base point $x \in L$ one has 
\begin{equation*}
    \nu\left(\{ u \in B : g_tux \in Q\}\right) > \nu(B) \mu(Q) - \varepsilon,
\end{equation*}
where $\nu$ is a Haar measure on $U$. The desired  result   immediately follows from the above measure estimate by taking $L = \{x\}$ and  $0 < \varepsilon < \nu(B) \mu(Q)$.\end{proof}

We now prove Theorem \ref{thm:dynamical_density_general}.

\begin{proof}[Proof of Theorem \ref{thm:dynamical_density_general}] Let $x_0 \in X$ with $\lgr^F_\phi(x_0) > 0$, and let $I \subset \R_{>0}$ be an open interval containing $\lgr^F_\phi(x_0)$.
We know from  Lemma~\ref{ZF} that  $\lgr^F_\phi$ %is continuous 
is   continuous along the orbits of $Z$. Hence there exists an open neighborhood $B^0 \subset Z$ of the identity such that 
\[
\lgr^F_\phi(z x_0) \in I \quad \text{for all } z \in B^0.
\]
By Lemma~\ref{uminus}, $\lgr^F_\phi$ is $U^-$-invariant. Thus
%, for any %open neighborhood 
%$B^- \subset U^-$ of the identity, 
we have
\[
\lgr^F_\phi(u^- z x_0) = \lgr^F_\phi(z x_0) \in I \quad \text{for all } u^- \in U^-, \, z \in B^0.
\]
Fix bounded open   neighborhoods $B$ and $B^-$ of the identity in $U$ and $U^-$ respectively and shrink $B^0$ if necessary to ensure that the multiplication map \equ{diffeo} is a diffeomorphism, and the map $B \times B^-\times B^0\to X$ given by $(u,u^-,z)\mapsto uu^-zx_0$ is injecttve. Define $$Q := B B^- B^0 x_0.$$ Choosing $B, B^0, B^-$ with piecewise smooth boundary one can ensure that $Q$ is a bounded open neighborhood of $x_0$ in $X$ with $\mu(\partial Q) = 0$.

Now let $x \in X$, and let $W \subset U$ be an arbitrary non-empty open set. Choose any $h_0 \in W$,
% and a small open neighborhood $V_0 \subset U$ of the identity such that $V_0 h_0 \subset W$. Also 
let $W_0$ be an open neighborhood  of the identity in $U$ such that $\overline{W_0}h_0 \subset W$, and choose a
neighborhood of identity $N \subset U$ with
$NW_0h_0 \subset W$.

\ignore{\begin{figure}[h]
\centering
\begin{tikzpicture}[scale=1.2]
    % Outer set W (Open neighborhood in U)
    \draw[thick, fill=blue!5, dashed] (0,0) ellipse (3.8cm and 2.8cm);
    \node at (-3.0, 2.0) {\Large $W$};

    % Thickened set N W_0 h_0
    \draw[thick, fill=purple!15, draw=purple!80!black, dash dot] (0.3,-0.2) ellipse (2.5cm and 1.7cm);
    \node[purple!80!black] at (-1.4, 1.0) {\large $N W_0 h_0$};

    % Shifted identity neighborhood W_0 h_0
    \draw[thick, fill=green!20, draw=green!50!black] (0.5,-0.4) ellipse (1.4cm and 0.9cm);
    \node[green!50!black] at (-0.2, 0.1) {$W_0 h_0$};

    % Basepoint h_0
    \filldraw (0.5,-0.4) circle (2pt) node[below right] {$h_0$};
    
\end{tikzpicture}
\caption{An illustration of the containment $W_0 h_0 \subset N W_0 h_0 \subset W$.}
\end{figure}}

Since $W_0h_0 \subset U$ and $Q \subset X$ are bounded non-empty open sets with $\mu(\partial Q) = 0$, Proposition~\ref{equidistr} guarantees that for all  sufficiently large ${t} > 0$ one has 
\[
g_{{t}} W_0 h_0 x \cap Q \ne \varnothing;
\]
that is, there exist $h \in W_0$ and $x' \in Q$ such that $g_{{t}} h h_0 x = x'$. 

The next step  is based on a shadowing argument that is standard in hyperbolic dynamics (see e.g.\ \cite{EKR}, where it was recently used to construct orbits on \hs s that are bounded and approach a given point). Write $x' = u   u^- zx_0$ for some $u  \in B $,  $u^- \in B^-$ and $z \in B^0$. Rearranging $g_{{t}} h h_0 x = u  u^- z x_0$, we %multiply by $g_{-{t}}$ on the left to
 obtain 
%\[
%h h_0 x = (g_{-{t}} u  g_{{t}}) g_{-{t}} z u^- x_0,
%\]
%or, equivalently,
\eq{conj}{
u^{-1} g_{{t}}h h_0   x =  g_{{t}}   \left(g_{-{t}} u^{-1} g_{{t}}\right) h h_0   x = u^- z x_0.
}
%Define $u := \big(g_{-{t}} (u^+)^{-1} g_{{t}}\big) h h_0 \in U$. 
Because the conjugation by $g_{{-t}}$ contracts $U$ as $t \to \infty$,
% (i.e., $g_{-{t}} (u^+)^{-1} g_{{t}} \to e$), 
by choosing ${t}$ sufficiently large %$g_{-{t}} u^{-1} g_{{t}}$ can be made arbitrarily close to the identity in $U$. In particular, 
we can achieve $g_{-{t}} B^{-1} g_{{t}}\subset N$, hence
$
\big(g_{-{t}} u^{-1} g_{{t}}\big) h \in NW_0. 
$
%for all $u \in B$ and $h \in W_0$. 
The latter implies that \eq{conj2}{u' := \big(g_{-{t}} u^{-1} g_{{t}}\big) h h_0 \in NW_0 h_0 \subset W.}
Applying $g_{{t}}$ and using the $F$-invariance of $\lgr^F_\phi$ together with Lemma~\ref{uminus}, we obtain
\[
\lgr^F_\phi(u' x) = \lgr^F_\phi(g_{{t}} u' x) \underset{\equ{conj2}}= \lgr^F_\phi(u^{-1} g_{{t}}h h_0x) \underset{\equ{conj}}= \lgr^F_\phi(u^- z x_0) %= \lgr^F_\phi\big((z u^- z^{-1}) zx_0\big) 
= \lgr^F_\phi(z x_0) \in I.
\]
%where %$z' = (u^-)^{-1} z u^- \in B^0$. 
%the last equality follows from the fact that $Z$ normalizes $U^-$, hence the multiplication of $z x_0$ by $z u^- z^{-1}\in U^-$ does not change the value of $\lgr^F_\phi$. 
Thus we have exhibited $u' \in W$ with $\lgr^F_\phi(u' x) \in I$, thereby establishing the density of $\mathbb{L}^F_\phi(Wx)$  in $\left[0, L^F_\phi\right]$.
%In particular, we have shown that for any non-empty open set $B \subset U$ and any $x\in X_d$   the restricted dynamical \ls\  $$\mathbb{L}^F_\phi(B x) = \left\{\lgr^F_\phi(u x) : u \in B\right\}$$ contains values in $I$. Since $I$ was chosen to be an arbitrary neighborhood of any non-zero value in $
%\mathbb{L}^F_\phi$, %and since $0 \in {\mathbb{L}^F_\phi(B x)}$ in view of the density of  $g_t$-trajectories of $u x$ for $\nu$-a.e.\ $u\in U$, 
%it follows that $\mathbb{L}^F_\phi(B x)$ is dense in $[0, L^F_\eta]$.
\end{proof}

  \section{Diophantine spectrum}\label{dioph}
 \subsection{The correspondence}\label{1.2} 
In this subsection we give a proof of   Proposition  \ref{thm:dynamical_formula}, which is %simply 
a quantitative version of Dani's correspondence from \cite{dani}.  Several variations already exist in the literature, such as %a correspondence 
with `inf' in place of `liminf' \cite[Lemma 3.1]{BK}, with `limsup' in place of `liminf' (this corresponds to the Dirichlet spectrum, see \cite[Lemma 2]{AW}), and for $(m,n) = (2,1)$ \cite[Lemma 2.1]{CWZ}. Still in order to make the paper self-contained we present a short complete proof here.  Recall that the set-up of Proposition  \ref{thm:dynamical_formula} involves $G = \SL_d(\R)$, $\Gamma = \SL_d(\Z)$, the action of $F$ as in  \equ{defF} on $X_d = \ggm$,   $\phi:X_d\to \R_{>0}$ defined as in \equ{defeta}--\equ{defdelta}, and lattices $x_A =  \begin{pmatrix}
   I_m  & A \\
    0 & I_n 
       \end{pmatrix}\Z^d$, where \amr.

\begin{proof}[Proof of Proposition  \ref{thm:dynamical_formula}]
First note that $g_t x_A $ consists of vectors of the form
	\begin{equation*}
	\begin{pmatrix}
   e^{t/m }I_m  & 0 \\
    0 & e^{-t/n }I_n 
       \end{pmatrix}
       \begin{pmatrix}
   I_m  & A \\
    0 & I_n 
       \end{pmatrix}
       \begin{pmatrix}
   -\p \\
    \vq
       \end{pmatrix}
      =
       \begin{pmatrix}
       e^{t/m }(A\vq - \p)  \\
       e^{-t/n }\vq       
       \end{pmatrix}.
           \end{equation*}
       Suppose that $\lgr_{\phi}^F({x_A})  < c$; then there exists an unbounded set of $t > 0$ such that \eq{equiv}{\left\| \begin{pmatrix}
       e^{t/m }(A\vq - \p)  \\
       e^{-t/n }\vq       
       \end{pmatrix}\right\|^d < c\underset{\equ{defeta}}\iff \begin{cases}\rho(A\vq - \p) ^m &< e^{-t }c^{m/d}\\\qquad\sigma(\vq ) ^n &< e^{t }\ \,c^{n/d}.\end{cases}}
       for some $\begin{pmatrix}
   \p \\
    \vq
       \end{pmatrix}
\in\Z^{d}\nz$. This leaves two alternatives: either $A\vq - \p = 0$, in which case clearly $\lgr_{\rho,\sigma}(A) =0$, or  \equ{equiv} holds for infinitely many $\vq$. The latter, after multiplying the two inequalities in the right hand side of  \equ{equiv}, implies that \eq{<c}{\dist_\rho(A\mathbf{q}, \Z^m) ^m  \sigma(\vq ) ^n < c,} hence $\lgr_{\rho,\sigma}(A)  \le c$. This proves that $\lgr_{\rho,\sigma}(A)  \le \lgr_{\phi}^F({x_A})$.

Conversely, assume that $\lgr_{\rho,\sigma}(A)  < c$, that is, there exist infinitely many $\vq\in\Z^n$ such that \equ{<c} holds. Again there are two cases. If $A\vq \in\Z^m $ for some $   \vq\in\Z^n\nz$ , then clearly $g_tx_A\to\infty$ as $t\to\infty$, hence $\lgr_{\phi}^F({x_A})  =0$. Otherwise denote 
 \eq{tq}{t_{\mathbf{q}} := \frac{mn}{d} \log\left(\frac{\sigma(\mathbf{q})}{\dist_\rho(A\mathbf{q}, \Z^m)}\right);}
 it is clear from  the finiteness of the covering radius of $\Z^m$ with respect to the norm $\rho$ that $t_{\mathbf{q}} \to \infty$ as $\mathbf{q} \to \infty$. Observe that  \equ{tq} is equivalent to 
 $$\vre_\vq := e^{t_{\mathbf{q}}/m }\dist_\rho(A\mathbf{q}, \Z^m) =  e^{-t_{\mathbf{q}}/n } \sigma(\mathbf{q});$$
thus, in view of  \equ{<c}, one has $\vre_\vq < c^{1/d}$. This forces \equ{equiv} to hold for each $\vq$ as above and some $\p\in\Z^m$, which implies $\lgr_{\phi}^F({x_A})  \le c$, hence $\lgr_{\rho,\sigma}(A)  \ge \lgr_{\phi}^F({x_A})$.
 \end{proof}

 \subsection{Diagonal actions on $X_d$}\label{sld} 
For the next step let us take $F = \{g_t: t \ge 0\} \subset G$ of the form  \eq{diaggt}{g_t = \diag\left( e^{w_1 t},\dots, e^{w_d t}\right),\text{ where }\sum_{i=1}^dw_i = 0.}
%where $w_1,\dots,w_d\in\R$ are such that $\sum_{i=1}^dw_i = 0$.
% $w_i \ne 0$ for some $i = 1,\dots,d$, and
%\begin{itemize}
%\item[\rm (i)] $w_i \ne 0$ for all $i$, and 
%\item[\rm (ii)]  
%\eq{equaleigenvalues}{w_i = w_j\text{ for some }i\ne j.}
%: namely $$g_t = \diag\left( e^{w_1 t},\dots, e^{w_d t}\right),$$
%where $w_1,\dots,w_d\in\R$ are such that \equ{equaleigenvalues} holds.
%\begin{itemize}
%\item[\rm (i)] $w_i \ne 0$ for all $i$, and 
%\item[\rm (ii)]  $w_i = w_j$ for some $i\ne j$.
%\end{itemize}
Clearly in this case the group $Z$ defined in \equ{defZ} coincides with  the identity component  of the  centralizer of $F$. The next lemma implies that $Z$ will satisfy \equ{unbdd}  as long as it contains a copy of $\SL_2(\R)$.
%A useful feature of our choice of $F$  is that  $Z$ is quite large. Indeed, \equ{equaleigenvalues} implies that a copy of $\SL_2(\R)$ acting in the plane spanned by $i$th and $j$th standard basis vectors of $\R^d$  commutes with $F$. This is used for the proof of the following lemma:

\begin{lemma}\label{commuting} Let $F$ be as in  \equ{diaggt}, and suppose that \eq{equaleigenvalues}{w_i = w_j\text{ for some }i\ne j.}
Then for any $\mathbf{v}\in\R^d\nz$ and any $s > 0$   there exists ${h}_s\in Z$  such that ${h}s\mathbf{v} = s\mathbf{v}$.%\end{itemize}
\end{lemma} 

\begin{proof}
By \equ{equaleigenvalues}, there exist distinct indices $i, j \in \{1, \dots, d\}$ such that $w_i = w_j$. Renumbering coordinates if necessary,
% and adjusting the norm $\eta$ accordingly, 
we may assume without loss of generality that $w_1 = w_2 = w$. 

Write $\mathbf{v} = (v_1, v_2, \dots, v_d)^T \in \R^d \nz$. We consider two cases based on the first two components of $\mathbf{v}$:

\smallskip
\noindent\textbf{Case 1:} 
 $\mathbf{u} := \begin{pmatrix} v_1 \\  v_2 \end{pmatrix}\ne \begin{pmatrix} 0 \\  0 \end{pmatrix}$. Complete $\mathbf{u}$ to a basis $(\mathbf{u}, \mathbf{u}')$ of $\R^2$, %where $\mathbf{u}' := \begin{pmatrix} v'_1 \\  v'_2 \end{pmatrix}$, 
 and  for any  $s > 0$ let $\theta_s \in \GL_2(\R)$ be the linear map defined by
\eq{deftheta}{
\theta_s \mathbf{u} = \mathbf{u} \quad \text{and} \quad \theta_s \mathbf{u}' = s^{-d} \mathbf{u}'.
}
\noindent\textbf{Case 2:} $(v_1, v_2) = (0, 0)$. 
In this case we can arbitrarily choose a basis $(\mathbf{u}, \mathbf{u}')$ of $\R^2$ and define $\theta_s \in \GL_2(\R)$ by \equ{deftheta};   for instance, let $\theta_s = \diag(1, s^{-d})$. 
%Then the matrix $h_s := s h_\theta$ commutes with $F$, lies in $\SL_d(\R)$, and satisfies $h_s \mathbf{v} = s \mathbf{v}$ because $\theta (0,0)^T = (0,0)^T$; hence   the conclusions of the lemma hold.
\smallskip

Clearly in both cases we have $\det \theta_s = s^{-d}$. Embed $\theta_s$ into $M_{d,d}(\R)$ as a block-diagonal matrix:
\[
{h}'_s := \begin{pmatrix} \theta_s & 0_{2 \times (d-2)} \\ 0_{(d-2) \times 2} & I_{d-2} \end{pmatrix} \in \GL_d(\R).
\]
Let $(\mathbf{e}_1,\dots, \mathbf{e}_d)$ denote the standard basis of $\R^d$.
Because $w_1 = w_2 = w$,   $g_t$ acts as $e^{wt} I_2$ on the subspace $\Span\{\mathbf{e}_1, \mathbf{e}_2\}$, hence ${h}'_s$ commutes with $F$.
Now define $${h}_s := s {h}'_s \in \GL_d(\R)$$ and check that ${h}_s$ satisfies all the required conditions. Clearly $$\det {h}_s %= \det(s z') 
= s^d \det {h}'_s = s^d \det \theta_s = 1,$$ thus ${h}_s\in G$. Also, by construction, %$z$ has a one-dimensional eigenspace spanned by 
%$\mathbf{v}' := \big((\mathbf{u}')^T, 0, \dots, 0\big)^T$ with eigenvalue $s^{-(d-1)}$  and  acts as multiplication by $s$ on the span of $\mathbf{e}_3,\dots, \mathbf{e}_d$ and $(\mathbf{u}^T, 0, \dots, 0)^T$; in particular, we have 
${h}_s \mathbf{v}= s \mathbf{v}$. Finally, ${h}_s\in Z$ since $s I_d$ lies in the center of $\GL_d(\R)$, ${h}'_s$ centralizes $F$, and $\{{h}_s\}$ is a continuous path with ${h}_1 = e$.
% we have $$h_s g_t = (s h'_s) g_t = g_t (s h'_s) = g_t h_s$$ for all $t \ge 0$. Thus $h_s \in Z$, 
This completes the proof. \end{proof}

\begin{corollary}\label{diaglagrange} Let $F$ as in  \equ{diaggt} be nontrivial ($w_i \ne 0$ for some $i = 1,\dots,d$) and such that  \equ{equaleigenvalues} holds. 
Also let $U$ be the \ehs\ of $G$ relative to $F$, and let $\phi : X_d \to \R_{> 0}$ be any bounded, continuous, proper function. Then
\begin{itemize}\item[\rm (i)]  \(
\mathbb{L}^F_\phi = \left[0,L^F_\phi\right]\); 

\smallskip
\item[\rm (ii)]  $\overline{\mathbb{L}^F_\phi(Wx)}  = \left[0, L^F_\phi\right]$ for any $x \in X$ and any non-empty open $W\subset U$. \end{itemize}
\end{corollary} 

\begin{proof} Lemma \ref{commuting} implies that for any $x\in X_d$, and any  $\mathbf{v}\in x\,\nz$ and any $s>0$ there exists  $h_s\in Z$ with $h_s\mathbf{v} \to 0$ as $s\to 0$; hence, in view of Mahler's Compactness Criterion,  $h_s x \to \infty$ as $s\to 0$. 
Thus the two parts of the corollary  are immediate from  Theorem \ref{thm:full_spectrum_general} (resp.~Theorem \ref{thm:dynamical_density_general}) and the ergodicity (resp.~mixing) of the $F$-action on $X_d$.\end{proof}

Finally, to prove Theorem \ref{thm1} it remains to take $F$ as in \equ{defF} and $\phi$ as in \equ{defeta}--\equ{defdelta}, and apply Proposition~\ref{thm:dynamical_formula}, Corollary \ref{diaglagrange}(ii) and the fact that $U$ as in \equ{defU} is   expanding horospherical relative to $F$.

  \section{Open questions}\label{open}
  \subsection{%Is $L_{\rho,\sigma}$ in ${\mathbb{L}_{\rho,\sigma}}$?
  The top value of the \ls}\label{top} The proof of Proposition \ref{topspectrum} does not extend to show that for a subset $Y\subset X_d$ the restricted {dynamical Lagrange spectrum}   $
\mathbb{L}^F_\phi(Y) := \left\{\lgr^F_\phi(x): x\in Y\right\}
$
contains its supremum: indeed, there is no reason for the limit lattice $x_0 \in X_d$ constructed in the course of the proof  to be an element of $Y$. In particular, when $\max(m,n) > 1$ it is an open question whether or not there exist norms $\rho$ on $\R^m$ and $\sigma$ on $\R^n$ such that $L_{\rho,\sigma}\in{\mathbb{L}_{\rho,\sigma}}$. 

 \subsection{Approximation with weights}\label{weights} It is tempting to apply the methods of this paper to \da\ with weights. Recall that, given an $m$-tuple 
$\mathbf{a} = (a_1, \ldots, a_m)$ and an $n$-tuple $\mathbf{b} = (b_1, \ldots, b_n)$  of positive real numbers satisfying $\sum\limits_{i=1}^m a_i = \sum\limits_{j=1}^n b_j = 1$, one defines quasi-norms 
$
\|\vx\|_{\va}:=\max_i|x_i|^{1/a_i} \quad\textrm{and}\quad \|\vy\|_{\vb}:=\max_j|y_j|^{1/b_j}
$
on $\R^m$ and $\R^n$ respectively. Then one says that \amr\ is \textsl{$(\va,\vb)$-badly approximable} if its weighted Lagrange constant $$\lgr_{\va,\vb}({A}) := \liminf_{ \mathbf{q}\to \infty} \|\mathbf{q}\|_\vb \min_{\vp\in\Z^m} \|{A}\mathbf{q}-\vp\|_\va $$
is positive. The choice of equal weights $\mathbf{a} = (1/m, \ldots, 1/m)$ and $\mathbf{b} = (1/n, \ldots, 1/n)$ yields the standard set-up of \da\ with respect to the supremum norms. A correspondence between approximation with weights $\mathbf{a}$, $\mathbf{b}$ and bounded trajectories of the semigroup  $${F^{\va,\vb} = \left\{g^{\va,\vb}_t: t \ge 0\right\} \subset G, \quad\text{where  }g^{\va,\vb}_t = \diag\left( e^{a_1 t},\dots, e^{a_m t}, e^{-b_1 t},\dots, e^{-b_n t}\right),}$$
on $X_d$ is well-known, see \cite{K98}. Corollary \ref{diaglagrange}(i) implies that the dynamical \ls\ with respect to 
$F^{\va,\vb}$ and any bounded, continuous, proper $\phi : X \to \R_{> 0}$    is a closed interval containing zero as long as at least two of the weights (either $a_i,a_j$ or $b_i,b_j$) coincide. However %even in this case 
this argument does not by itself yield a conclusion
%nothing can be derived 
about the weighted  \di\ \ls\  $\left\{\lgr_{\va,\vb}({A}): A\in\mr\right\}$, because it is crucially used in the proof that $U$ as in \equ{defU} is  expanding horospherical  relative to $F$.

 \subsection{No repeated eigenvalues}\label{multone} Condition \equ{equaleigenvalues} %is necessary 
cannot be omitted in
general
 for the conclusions of Corollary \ref{diaglagrange} to hold, as can be seen by taking $X = X_2$, $g_t = \diag\left( e^{t},  e^{- t}\right)$ and $$\phi(x) = \inf_{\mathbf{v}\in x\nz}\|\mathbf{v}\|^2,\text{  where }\|\cdot\| \text{ is the supremum norm on }\R^2.$$  In this case %the centralizer $Z$ of $F$ is equal to 
 $Z = \{g_t: t\in \R\}$, hence   $\lgr^F_\phi$ is $Z$-invariant, which makes $\mathbb{L}^F_\phi$ coincide with the  one-dimensional \ls\ $\mathbb{L}$. Similarly, if  $$F = \left\{\diag\left( e^{w_1 t},\dots, e^{w_d t}\right):t\ge 0\right\}\subset \SL_d(\R)$$ is such that all the values $w_i$ are different, the centralizer of $F$ is the group of diagonal matrices, and the argument used in \S\ref{dynlagr} to produce intermediate values of dynamical Lagrange constants does not apply. It is an interesting open question to determine the structure of ${\mathbb{L}^F_{\phi}}$ in this case.  In particular, when $d > 2$ it is not known whether or not    $L^F_{\phi}$ can be an isolated point of ${\mathbb{L}^F_{\phi}}$.

\end{document}